\documentclass[12pt]{amsart}
\usepackage{amsmath}
\usepackage{graphicx}
\usepackage{hyperref}
\usepackage{mathtools}
\usepackage{amsfonts}
\usepackage{amssymb}
\usepackage[T1]{fontenc}
\usepackage{amsthm}
\usepackage{fullpage}
\usepackage{enumitem}
\usepackage[T1]{fontenc}
\usepackage[utf8]{inputenc}
\usepackage{bm}

\newtheorem{theorem}{Theorem}[section]

\newtheorem{lemma}[theorem]{Lemma}

\newtheorem{proposition}[theorem]{Proposition}

\theoremstyle{definition}
\newtheorem{definition}{Definition}[section]

\theoremstyle{remark}
\newtheorem{remark}{Remark}

\numberwithin{equation}{section}

\title{ On Some Non-Permutations of Quadratic extensions Of Finite Field}

\keywords{Permutation Polynomial, Weil Sum, Finite Fields }
\subjclass[2020]{11T06, 11T23, 12E10}

\author{Bidushi Sharma}
\address{Department of Mathematical Sciences, Tezpur University, Tezpur, Assam, 784028, India}
\email{msp22102@tezu.ac.in}
\author{DHIREN KUMAR BASNET*}
\address{Department of Mathematical Sciences, Tezpur University, Tezpur, Assam, 784028, India}
\email{dbasnet@tezu.ernet.in}
\begin{document}
	\begin{abstract}

    In this article, we study polynomials over $\mathbb{F}_{q^2}$ that do not permute $\mathbb{F}_{q^2}$. More precisely, we characterize polynomials of the forms $x^q + b x^2 + c x + d$ and $x^{q+1} + b x^q + c x + d$ over $\mathbb{F}_{q^2}$ according to whether they are permutation or non-permutation polynomials. To this end, we determine the exact number of zeros of these polynomials using existing results on certain special Weil sums.
    \end{abstract}

	\maketitle
	
	\section{Introduction}
    A permutation polynomial (PP) over the finite field $\mathbb{F}_{q}$ is a polynomial that induces a bijection on $\mathbb{F}_{q}$. Permutation polynomials have attracted the interest of many researchers due to their wide applications in coding theory \cite{c2,c1}, cryptography \cite{cr1,cr2}, and combinatorial design \cite{cd}. For a comprehensive survey of recent advancements on this topic, we refer the readers to \cite{r5,p1,p2,p3}. Motivated by these works, we investigate conditions under which certain generalized polynomials fail to be permutations. In particular, we establish necessary and sufficient conditions on the coefficient parameters that determine whether these polynomials permute or fail to permute the quadratic extension $\mathbb{F}_{q^2}$ of $\mathbb{F}_{q}$.

The central idea used in this paper is that a polynomial $f$ permutes $\mathbb{F}_{q^2}$ if and only if the equation $f(x)=t$ has exactly one solution in $\mathbb{F}_{q^2}$ for every $t \in \mathbb{F}_{q^2}$. Equivalently, if there exists at least one $t \in \mathbb{F}_{q^2}$ such that the polynomial $f(x)+t$ has at least two zeros in $\mathbb{F}_{q^2}$, then $f$ does not permute $\mathbb{F}_{q^2}$. Thus, in the first part of this article, we focus on determining the exact number of zeros of certain special polynomials over $\mathbb{F}_{q^2}$. These root-counting results are then used to characterize the corresponding polynomials as permutations or non-permutations. More precisely, we study polynomials of the form $x^q + b x^2 + c x + d$ and $x^{q+1} + b x^q + c x + d$. Note that none of the polynomials considered here are of Wan-Lidl type, that is, of the form $x^s f(x^r)$ with $s,r \in \mathbb{N}$ and $r \mid (q-1)$, nor they are affine $q$-polynomial. It is worth noting that the polynomial $x^q + b x^2 + c x+d$, over fields of even characteristic, behaves as an affine $2$-polynomial, however, it is not affine $q$-polynomial.
    
    An important tool in determining the number of zeros of polynomials over $\mathbb{F}_{q^2}$ is the Weil sum. A Weil sum is a character sum over a finite field in which the character is evaluated at a polynomial argument. In our case, we use Weil sums associated with the canonical additive character, namely
$$
\sum_{x \in \mathbb{F}_{q^2}} \chi\big(f(x)\big),
$$
where $\chi$ is the canonical additive character of $\mathbb{F}_{q^2}.$
Let $M$ denote the number of zeros of $f(x)$ in $\mathbb{F}_{q^2}$. Then
$$
M = \frac{1}{q^2} \sum_{x \in \mathbb{F}_{q^2}} \sum_{y \in \mathbb{F}_{q^2}} \chi\big(y f(x)\big),
$$
which provides a convenient tool for determining $M$. Using character sums to study permutation behavior has been employed in various works to study permutation polynomials and permutation rational functions. For instance, in \cite{pp1}, the permutation behavior of affine $q$-polynomials of the form
$$
f(x) = a_t x^{q^t} + a_{t-1} x^{q^{t-1}} + \cdots + a_1 x^q + a_0 x + a
$$
is investigated. In \cite{pp2}, over fields of even characteristic, the permutation behavior of polynomials of the form $\operatorname{Tr}(A x^{q+1}) + L(x)$, where $L$ is $q$-linearized and $A\neq 0$, is studied. Further applications of this method can be found in \cite{r1}.

In Section 2, we collect the preliminary results required for our work. Here, $\operatorname{Tr}$ and $\operatorname{N}$ denote the trace and norm maps from $\mathbb{F}_{q^2}$ to $\mathbb{F}_q$. In Section 3, we begin by determining the number of roots of the polynomial over $\mathbb{F}_{q^2}$ and conclude by identifying explicit cases in which it does not permute $\mathbb{F}_{q^2}$.
    \section {Preliminaries}
    In Lemmas~\ref{l1}--\ref{l3}, we present all the Weil sums that arise in the process of determining the number of roots over both odd and even ordered fields.
    \begin{lemma}\cite{r1}\label{l1}
        Let $q$ be odd, $\chi$ denote the canonical additive character of $\mathbb{F}_{q^2}$, and let $\eta$ denote the quadratic character of $\mathbb{F}_q^{*}$. Let $A, B, C \in \mathbb{F}_{q^2}$ with $B \neq 0$, and define $D = \operatorname{Tr}(A)^2 - 4N(B)$. If $D \neq 0$, then  
$$
\sum_{w \in \mathbb{F}_{q^2}} \chi\big( Aw^{q+1} + Bw^2 + Cw \big)
= - \eta(D) \, q \, \chi\big( -A\theta^{q+1} - B\theta^2 \big),
$$
where  
$$
\theta = \frac{-\operatorname{Tr}(A)C^q + 2B^q C}{\operatorname{Tr}(A)^2 - 4N(B)},
$$
that is, $\theta$ is the unique root of the equation $\operatorname{Tr}(A)x^q + 2Bx + C = 0$. 
\end{lemma}
\begin{lemma}\cite{r2}\label{l2}
    Let $q$ be even and $\chi$ denote the canonical additive character of $\mathbb{F}_{q^2}$. Let $A, B \in \mathbb{F}_{q^2}$ with $A \neq 0$. Then
$$
\sum_{w \in \mathbb{F}_{q^2}} \chi\big( A w^{q+1} + B w \big) =
\begin{cases}
q^2, & \text{if } A \in \mathbb{F}_q^* \text{ and } B = 0, \\
-q, & \text{if } A \in \mathbb{F}_{q^2}\setminus\mathbb{F}_{q} \text{ and } B=0,
\\
0, & \text{if } A \in \mathbb{F}_q^* \text{ and } B \neq 0, \\
- q \, \chi\left( \dfrac{A B^{q+1}}{\operatorname{Tr}(A)^2} \right), & \text{if } A \in \mathbb{F}_{q^2} \setminus \mathbb{F}_q \text{ and }B\neq0.
\end{cases}
$$
\end{lemma}
Although Lemma~\ref{l2.3} is not explicitly stated as a standalone result in \cite{r2}, it can be deduced from the proof of Theorem~4.2 therein.
\begin{lemma}\cite{r2}\label{l2.3}
    Let $q$ be even and $\chi$ denote the canonical additive character of $F_{q^2}$. For $A,B\in F_{q^2}$, $A\neq 0$, we have
$$
\sum_{x\in F_{q^2}} \chi(Ax^2+Bx)
=
\begin{cases}
q^2, & \text{if } A=B^2,\\
0, & \text{otherwise}.
\end{cases}
$$
\end{lemma}
\begin{lemma}\cite{r3,r4}\label{l3}
    Let $q$ be odd and $\chi$ denote the canonical additive character of $\mathbb{F}_{q^2}$. Let $A, B \in \mathbb{F}_{q^2}$ with $A \neq 0$. Then
$$
\sum_{w \in \mathbb{F}_{q^2}} \chi\big(A w^{q+1} + B w\big) =
\begin{cases}
-q, & \text{if } \operatorname{Tr}(A) \neq 0 \text{ and } B = 0,\\[1mm]
-q \, \chi\Big( -\dfrac{A B^{q+1}}{\operatorname{Tr}(A)^2} \Big), & \text{if } \operatorname{Tr}(A) \neq 0 \text{ and } B \neq 0,\\[1mm]
q^2, & \text{if } \operatorname{Tr}(A) = 0 \text{ and } B = 0,\\[1mm]
0, & \text{if } \operatorname{Tr}(A) = 0 \text{ and } B \neq 0.
\end{cases}
$$
\end{lemma}

\begin{definition}
    A nonzero polynomial $C(x) \in \mathbb{F}_{q^2}[x]$ of degree $n$ is called a self-conjugate reciprocal polynomial if
$$
x^n C^{(q)}\!\left(\frac{1}{x}\right) = \beta C(x)
$$
for some $\beta \in \mathbb{F}_{q^2}$. Here, $C^{(q)}$ denotes the polynomial obtained by raising each coefficient of $C(x)$ to its $q$-th power.
\end{definition}
A comprehensive account of SCR polynomials is given in \cite{pp3}.
\begin{lemma}\cite{r5}\label{l4}
    Let $q$ be even and $C(x) = \alpha x^2 + \beta x + \alpha^q$ 
be a self-conjugate reciprocal polynomial, where $\alpha \in \mathbb{F}_{q^2}$ and $\beta \in \mathbb{F}_q^*$. Let $\mu_{q+1}=\{x \in \mathbb{F}_{q^2}^* : x^{q+1}=1\}.
$ Then the following statements are equivalent:  

\begin{enumerate}
    \item $C(x)$ has at least one root in $\mu_{q+1}$.
    \item $C(x)$ has two distinct roots in $\mu_{q+1}$.
    \item $\operatorname{Tr}_{\mathbb{F}_{q}}\!\left(\dfrac{\alpha^{q+1}}{\beta^2}\right)=1.
$
\end{enumerate}
\end{lemma}
The following two lemmas help us in giving the inverse of the permutation polynomials we encounter in this paper.

\begin{lemma}\cite{j1}\label{l2.6}
Let $L(X) = \sum_{i=0}^{n-1} \alpha_i X^{q^i}$ be a $q$-linearized permutation polynomial over $\mathbb{F}_{q^n}$, and let $D_L$ be its associated Dickson matrix. Then
$$
L^{-1}(X) = (\det D_L)^{-1} \sum_{i=0}^{n-1} \bar{\alpha_i} X^{q^i},
$$
where $\bar{\alpha_i}$ is the $(i,0)$-th cofactor of the matrix $D_L$. The Dickson matrix $D_L$ is given by
$$
D_L =
\begin{pmatrix}
\alpha_0 & \alpha_{n-1}^q & \alpha_{n-2}^{q^2} & \cdots & \alpha_1^{q^{n-1}} \\
\alpha_1 & \alpha_0^q     & \alpha_{n-1}^{q^2} & \cdots & \alpha_2^{q^{n-1}} \\
\alpha_2 & \alpha_1^q     & \alpha_0^{q^2}     & \cdots & \alpha_3^{q^{n-1}} \\
\vdots   & \vdots         & \vdots             & \ddots & \vdots \\
\alpha_{n-1} & \alpha_{n-2}^q & \alpha_{n-3}^{q^2} & \cdots & \alpha_0^{q^{n-1}}
\end{pmatrix}.
$$
\end{lemma}

\begin{lemma}\cite{j2}\label{l2.7}
Consider any polynomial $g \in \mathbb{F}_{q^n}[x]$, any additive polynomials $\phi,\psi \in \mathbb{F}_{q^n}[x]$, any $q$-polynomial $\overline{\psi} \in \mathbb{F}_{q^n}[x]$ satisfying $\phi \circ \psi = \overline{\psi} \circ \phi
$ and $|\psi(\mathbb{F}_{q^n})| = |\overline{\psi}(\mathbb{F}_{q^n})|$, and any polynomial $h \in \mathbb{F}_{q^n}[x]$ such that
$
h(\psi(\mathbb{F}_{q^n})) \subseteq \mathbb{F}_q \setminus \{0\}.
$
Let
$
f(x) = h(\psi(x))\,\phi(x) + g(\psi(x)) \text{ and } 
\bar{f}(x) = h(x)\phi(x) + \overline{\psi}(g(x)).
$ Assume that $f$ permutes $\mathbb{F}_{q^n}$, and further assume that $|S_{\psi}| = |S_{\overline{\psi}}|$ and
$\ker(\phi) \cap \psi(S_{\psi}) = \{0\}$. Then $\phi$ induces a bijection from $S_{\psi}$ to $S_{\overline{\psi}}$. Let $\bar{f}^{-1},\, \phi^{-1}\!\big|_{S_{\overline{\psi}}} \in \mathbb{F}_{q^n}[x]$ induce the inverses of $\bar{f}\!\big|_{\psi(\mathbb{F}_{q^n})}$ and $\phi\!\big|_{S_{\psi}}$, respectively. Then the compositional inverse of $f$ on $\mathbb{F}_{q^n}$ is given by
$$
f^{-1}(x)
=
\bar{f}^{-1}(\overline{\psi}(x))
+
\phi^{-1}\!\big|_{S_{\overline{\psi}}}
\left(
\frac{
x - \overline{\psi}(x) - g(\bar{f}^{-1}(\overline{\psi}(x))) + \overline{\psi}(g(\bar{f}^{-1}(\overline{\psi}(x))))
}
{h(\bar{f}^{-1}(\overline{\psi}(x)))}
\right).
$$
Furthermore, if $\phi$ induces a bijection from $\psi(\mathbb{F}_{q^n})$ to $\overline{\psi}(\mathbb{F}_{q^n})$, then $\phi$ is a permutation of $\mathbb{F}_{q^n}$ and the compositional inverse of $f$ on $\mathbb{F}_{q^n}$ is given by
$$
f^{-1}(x)
=
\phi^{-1}
\left(
\frac{
x - g(\bar{f}^{-1}(\overline{\psi}(x)))
}{
h(\bar{f}^{-1}(\overline{\psi}(x)))
}
\right).
$$
\end{lemma}
\section{Main Results}
\subsection*{Notation}

Throughout this section, we use the following notation:

\begin{itemize}
    \item $\psi$ denotes the canonical additive character of $\mathbb{F}_q$.
    
    \item $\chi$ denotes the canonical additive character of $\mathbb{F}_{q^2}$, defined by $\chi = \psi \circ \operatorname{Tr}$, where $\operatorname{Tr} : \mathbb{F}_{q^2} \to \mathbb{F}_q$ is the trace map given by $
    \operatorname{Tr}(x):=x^q+x.$
    
    \item $\eta$ denotes the quadratic character of $\mathbb{F}_q^{*}$.
    
    \item $\eta'$ denotes the quadratic character of $\mathbb{F}_{q^2}^{*}$, defined by $
    \eta'(x)=\eta(x^{q+1}).$
    \item $\mu_{q+1}:=\{x\in \mathbb{F}_{q^2}^*: x^{q+1}=1\}$, the set of $(q+1)$-th roots of unity.
    \item $\operatorname{Tr}_{\mathbb{F}_q}$
denotes the absolute trace function from $\mathbb{F}_{q}$ to $\mathbb{F}_{2}$.
\end{itemize}
We first explicitly determine the number of zeros of the polynomials under consideration. In particular, Lemmas~\ref{l5} and \ref{l6} are employed to analyze the polynomial $x^q + b x^2 + c x + d$ for odd $q$.
\begin{lemma}\label{l5}
Let $q=p^k$ and let $c,d \in \mathbb{F}_{q^2}$. Then $M$, the number of zeroes of $
x^q+cx+d$
in $\mathbb{F}_{q^2}$ is given by
$$
M=\begin{cases}
1, & \text{if } c \notin \mu_{q+1},\\[1mm]
q, & \text{if } c \in \mu_{q+1} \text{ and } d=cd^q,\\[1mm]
0, & \text{if } c \in \mu_{q+1} \text{ and } d \neq cd^q.
\end{cases}
$$
\end{lemma}

\begin{proof}
\begin{align*}
Mq^2
&= \sum_{w \in \mathbb{F}_{q^2}} \sum_{u \in \mathbb{F}_{q^2}} \chi\big(u f(w)\big) \\
&= \sum_{w \in \mathbb{F}_{q^2}} \sum_{u \in \mathbb{F}_{q^2}} \chi(uw^q + cuw)\chi(du) \\
&= \sum_{w \in \mathbb{F}_{q^2}} \sum_{u \in \mathbb{F}_{q^2}} \chi\big((u^q+cu)w\big)\chi(du) 
\quad \text{(since $\chi(u)=\chi(u^q)$ for all $u \in \mathbb{F}_{q^2}$)} \\
&= q^2 + \sum_{u \in \mathbb{F}_{q^2}^{*}} \chi(du)\sum_{w \in \mathbb{F}_{q^2}} \chi\big((u^q+cu)w\big).
\end{align*}

First consider the case $c \notin \mu_{q+1}$. Then for any $u \neq 0$, we have
$u^q+cu \neq 0,$
and hence $Mq^2=q^2.$

Now suppose that $c \in \mu_{q+1}$, and let $N$ be the set of zeros of the $q$-polynomial $u^q+cu$ in $\mathbb{F}_{q^2}$. Then $N$ is a $\mathbb{F}_{q}$-subspace of $\mathbb{F}_{q^2}$ of dimension $1$, i.e.,
$N=\{\alpha u : u \in \mathbb{F}_q\},
$ where $\alpha \in \mathbb{F}_{q^2}^{*}$ satisfies
$ \alpha^q+c\alpha=0.$

Again,
\begin{align*}
Mq^2
&= q^2 + \sum_{u \in \mathbb{F}_{q^2}^{*}\cap N} \chi(du)\sum_{w \in \mathbb{F}_{q^2}} 1
   + \sum_{u \in \mathbb{F}_{q^2}^{*}\setminus N} \chi(du)\sum_{w \in \mathbb{F}_{q^2}} \chi\big((u^q+cu)w\big) \\
&= q^2 + q^2 \sum_{u \in N\setminus\{0\}} \chi(du) \\
&= q^2 + q^2 \sum_{u \in \mathbb{F}_q^{*}} \chi(d\alpha u) \\
&= q^2 + q^2 \sum_{u \in \mathbb{F}_q^{*}} \psi\big(u\,\operatorname{Tr}(d\alpha)\big) \\
&=
\begin{cases}
q^2+q^2(q-1), & \text{if } \operatorname{Tr}(d\alpha)=0,\\[1mm]
q^2-q^2, & \text{if } \operatorname{Tr}(d\alpha)\neq 0.
\end{cases}
\end{align*}

Now, since $\alpha^q=-c\alpha$, we have $\operatorname{Tr}(d\alpha)=0
\quad \text{if and only if} \quad
d=cd^q.$
Therefore, the result follows.
\end{proof}

\begin{lemma}\label{l6}
Let $q$ be odd, $b,d \in \mathbb{F}_{q^2}$ with $b \neq 0$, and $\alpha$ be a primitive element of $\mathbb{F}_{q^2}$. Then the number of zeros of
$x^q+bx^2+d$ in $\mathbb{F}_{q^2}$ is
$$
1-\eta(-b^{q+1}) \sum_{i \in K} (-1)^i,
$$
where
$K=\{\, i \mid 0 \leq i \leq q,\ h(\alpha^{i(q-1)})=0 \,\},
$ and $h(x)=b^q x^3-4b^{q+1}d^q x^2-4b^{q+1}d x+b$.

\end{lemma}
\begin{proof}
Let $M$ denote the number of zeros of $f(x)=x^q+bx^2+d$ in $\mathbb{F}_{q^2}$. Then

\begin{align*}
Mq^2 
&= \sum_{u \in \mathbb{F}_{q^2}} \sum_{w \in \mathbb{F}_{q^2}} \chi\big(u f(w)\big) \\
&= \text{ 
}q^2 + \sum_{u \in \mathbb{F}_{q^2}^{*}} \chi(du) \sum_{w \in \mathbb{F}_{q^2}} \chi\big(b u w^2 + u^q w\big).
\end{align*}

Applying Lemma \ref{l1} to the inner sum, we have $D=-4b^{q+1}u^{q+1}=0$ if and only if $u=0$. Hence, for $u \neq 0$,  

\begin{align*}
\sum_{w \in \mathbb{F}_{q^2}} \chi\big(b u w^2 + u^q w\big) 
&= -\eta(-4 b^{q+1} u^{q+1}) \, q \, \chi\big(-b u \theta_u^2\big) \\
&= -\eta(-b^{q+1}) \, q \, \chi\left(-\frac{u^{2q-1}}{4b}\right) \eta'(u),
\end{align*}
where $\theta_u=-\dfrac{u^q}{2bu}$.
 Therefore,
\begin{align*}
Mq^2
&= q^2-q\eta(-b^{q+1})\sum_{u \in \mathbb{F}_{q^2}^{*}} \chi\left(\frac{4bdu-u^{2q-1}}{4b}\right)\eta'(u).
\end{align*}

Now every nonzero element of $\mathbb{F}_{q^2}$ can be written in the form $\alpha^i u$, where $0 \leq i \leq q$ and $u \in \mathbb{F}_q^{*}$. Also,
$
\eta'(\alpha^i u)=\eta'(\alpha)^i.
$
Since $\alpha$ is a primitive element of $\mathbb{F}_{q^2}$, we have
$
\eta'(\alpha)=-1.
$
Hence,
\begin{align*}
Mq^2
&= q^2-\eta(-b^{q+1})q\sum_{i=0}^{q}(-1)^i\sum_{u \in \mathbb{F}_q^{*}}
\chi\left(\frac{4bd\alpha^i-(\alpha^i)^{2q-1}}{4b}u\right) \\
&= q^2-\eta(-b^{q+1})q\sum_{i=0}^{q}(-1)^i\sum_{u \in \mathbb{F}_q^{*}}
\psi\left(u\,\operatorname{Tr}\left(\frac{4bd\alpha^i-(\alpha^i)^{2q-1}}{4b}\right)\right).
\end{align*}
Let $K\subseteq\{0,1,\dots,q\}$ such that
$
\operatorname{Tr}\left(\dfrac{4bd\alpha^i-(\alpha^i)^{2q-1}}{4b}\right)=0
\quad \text{if and only if} \quad i \in K.
$
This gives
$$
\sum_{u \in \mathbb{F}_q^{*}} \psi\left(u\,\operatorname{Tr}\left(\frac{4bd\alpha^i-(\alpha^i)^{2q-1}}{4b}\right)\right)
=
\begin{cases}
q-1, & \text{if } i \in K,\\
-1, & \text{if } i \notin K.
\end{cases}
$$

Hence,
\begin{align*}
Mq^2
&= q^2-\eta(-b^{q+1})q\left[\sum_{i \in K}(-1)^i(q-1)-\sum_{i \notin K}(-1)^i\right] \\
&= q^2-\eta(-b^{q+1})q^2\sum_{i \in K}(-1)^i.
\end{align*}

Furthermore,
$$
\operatorname{Tr}\left(\frac{4bd\alpha^i-(\alpha^i)^{2q-1}}{4b}\right)=0
$$
if and only if
$$
4b^{q+1}d\alpha^i-(\alpha^i)^{2q-1}b^q+4b^{q+1}d^q\alpha^{iq}-(\alpha^i)^{2-q}b=0,
$$
which is equivalent to
$$
4b^{q+1}d-b^q x^2+4b^{q+1}d^q x-bx^{-1}=0,
$$
where $x=\alpha^{i(q-1)}$. Hence the result follows.
\end{proof}
In Lemma~\ref{l6}, one might question how the evaluation of the character sum differs from that in \cite{r1}. We note that in \cite{r1}, in all cases the discriminant is of the form $D = (u^q - u)^2$, which implies that $D \neq 0$ if and only if $u \in \mathbb{F}_{q^2} \setminus \mathbb{F}_q$, and in this case $\eta'(D) = 1$ always. 

However, in our setting, $D = 0$ if and only if $u = 0$. Consequently, the evaluation of $\eta'(D)$ for $u \in \mathbb{F}_{q^2} \setminus \mathbb{F}_q$ presents additional difficulties. To overcome this, we employ a different representation of $\mathbb{F}_{q^2}$ than that used in \cite{r1}.

In Lemmas~\ref{l3.3}-\ref{l7}, we consider the case where $q$ is even for two different polynomials, while in Lemma~\ref{l3.3} we treat the case where $q$ is odd. Each problem is addressed using number theoretic arguments or purely character sum estimates.

\begin{lemma}\label{l3.3}
Let $q=2^k$ and $b,c \in \mathbb{F}_q$ with $b \neq 0$. Let $\chi$ be the canonical additive character of $\mathbb{F}_{q^2}$. Let $M$ denote the number of zeroes of 
$
f(x)=x^q+bx^2+cx+d \in \mathbb{F}_{q^2}[x]
$ in $\mathbb{F}_{q^2}$. Then the following hold

\begin{enumerate}
\item If $c=0$ and $k$ is even, then
$$
M=1+\chi(dx_1),
$$
where $x_1$ is the unique root of 
$
x^{2q-1}=b
$ 
in $\mathbb{F}_{q^2}$ 

\item If $c=0$, and $k$ is odd, then
$$
M
=
1+\chi(dx_1)+\chi(dx_2)+\chi(dx_3),
$$
where $x_1,x_2,x_3$ are the three distinct roots of $x^{2q-1}=b$ in $\mathbb{F}_{q^2}$ .
\item If $c=1$, then
$
M=1.
$

\item If $c \in \mathbb{F}_q^* \setminus \{1\}$ with
$
\operatorname{Tr}_{\mathbb{F}_q/\mathbb{F}_2}\!\left(\dfrac{1}{(c+1)^4}\right)=0,
$
then
$$
M=1+\chi\!\left(\dfrac{db}{(c+1)^2}\right).
$$

\item If $c \in \mathbb{F}_q^* \setminus \{1\}$ with
$
\operatorname{Tr}_{\mathbb{F}_q/\mathbb{F}_2}\!\left(\dfrac{1}{(c+1)^4}\right)=1,
$
then
$$
M
=
1
+
\chi\!\left(\frac{db}{(x_1+c)^2}\right)
+
\chi\!\left(\frac{db}{(x_2+c)^2}\right)
+
\chi\!\left(\frac{db}{(c+1)^2}\right),
$$
where $x_1,x_2$ are the two distinct roots of
$
x^2+(c+1)^2x+1=0
$ in $\mu_{q+1}\setminus\{1\}$.

\end{enumerate}
\end{lemma}
\begin{proof}
We begin with
$$
Mq^2
=
q^2+\sum_{u\in \mathbb{F}_{q^2}^*}\chi(du)\sum_{w\in \mathbb{F}_{q^2}}
\chi\!\left(buw^2+(u^q+cu)w\right).
$$
By Lemma~\ref{l2.3}, the inner sum equals $q^2$ if
$
u^{2q}+c^2u^2+bu=0,
$
and equals $0$ otherwise. Hence, if we define
$$
H=\left\{u\in \mathbb{F}_{q^2}^* : u^{2q}+c^2u^2+bu=0\right\},
$$
then
$$
M=1+\sum_{u\in H}\chi(du).
$$
Therefore, our next goal is to study the set $H$.
First suppose that $c=0$. Then we need to determine the solutions of
$
u^{2q-1}=b
$
in $\mathbb{F}_{q^2}$. Let $\alpha$ be a primitive element of $\mathbb{F}_{q^2}$ and write
$
b=\alpha^s
$
for some integer $s$. We need to determine the number of integers $r$ with
$
0 \leq r \leq q^2-2
$
such that $u=\alpha^r$ satisfies
$
\alpha^{r(2q-1)}=\alpha^s,
$
i.e.,
$
r(2q-1)\equiv s \pmod{q^2-1}.
$ This congruence has a solution if and only if
$
\gcd(2q-1,q^2-1)\mid s,
$
and the number of such integers $r$ is exactly
$
\gcd(2q-1,q^2-1).
$
Since $q=2^k$ with $k$ even, we have
$
\gcd(2q-1,q^2-1)=1\mid s.
$
Hence for every $b\in \mathbb{F}_{q^2}^*$, the equation
$
u^{2q-1}=b
$
has a unique solution $x_1\in \mathbb{F}_{q^2}^*$. Similarly, when $k$ is odd, we have
$
\gcd(2q-1,q^2-1)=3.
$
Since $\mathbb{F}_q^*=\langle \alpha^{q+1}\rangle \subseteq \langle \alpha^3\rangle,
$
it follows that for every $b\in \mathbb{F}_q^*$, the equation
$
u^{2q-1}=b
$
has exactly three distinct solutions $x_1,x_2,x_3\in \mathbb{F}_{q^2}.
$

Now when $c=1$, we examine the number of zeroes of 
$
u^{2q}+u^2+bu
$ in $\mathbb{F}_{q^2}$.
Let
$
r=bu$. Therefore
$ u^{2q}+u^2=\operatorname{Tr}(u)^2=r\in\mathbb{F}_{q}$ . Substituting $u=r/b$, we obtain
$
\operatorname{Tr}\!\left(\dfrac{r}{b}\right)^2=r.
$
Since $r\in \mathbb{F}_q$, this yields
$
r=0.
$
Hence $u=0$ and consequently $H=\phi$. 

Now, let $c\in \mathbb{F}_q^*\setminus\{1\}$. Then $u^{2q-1}+c^2u+b=0 
\iff u^{2(q-1)}u+c^2u+b=0$. Let $u^{q-1}=r\in \mu_{q+1}$. We obtain $(r+c)^2u=b$ and since $c\in \mathbb{F}_q^*\setminus\{1\}$, we obtain
\begin{align}\label{eq1}
u=\frac{b}{(r+c)^2}.
\end{align}

Equation $\ref{eq1}$ together with the fact that $u^{q-1}=r$, yields
\begin{align*}
&r^3+c^2r^2+c^2r+1=0\\
\implies& (r+1)\bigl(r^2+(c+1)^2r+1\bigr)=0. 
\end{align*}
Notice that
$
r^2+(c+1)^2r+1\neq 0
$
when $r=1$. Moreover, by Lemma \ref{l4}, the polynomial
$
r^2+(c+1)^2r+1
$
has two distinct roots $x_1,x_2\in \mu_{q+1}\setminus \{1\}$ if
$
\operatorname{Tr}_{\mathbb{F}_q/\mathbb{F}_2}\left(\dfrac{1}{(c+1)^4}\right)=1,
$
and has no root in $\mu_{q+1}\setminus\{1\}$ otherwise. Since
$
u=\dfrac{b}{(r+c)^2},
$
the possible values of $u$ satisfying $u^{2q-1}+c^2u+b=0$ are
$$
\frac{b}{(c+1)^2}, \quad \frac{b}{(x_1+c)^2}, \quad \frac{b}{(x_2+c)^2}.
$$
Therefore, the result follows.
\end{proof}    

\begin{lemma}\label{l7}
Let $q$ be even, and let $\alpha$ be a primitive element of
$\mathbb{F}_{q^2}$. Then $M$, the number of zeros of
$f(x)=x^{q+1}+bx^q+cx+d\in\mathbb{F}_{q^2}[x]$
in $\mathbb{F}_{q^2}$ is given by
\[
M=
\begin{cases}
1, & b=c=0,\ d=0,\\[2mm]
q+1, & b=c=0,\ d\in\mathbb{F}_q^*,\\[2mm]
0, & b=c=0,\ d\notin\mathbb{F}_q,\\[2mm]
1, & b^q=c\neq0,\ d\in\mathbb{F}_q,\ d=c^{q+1},\\[2mm]
q+1, & b^q=c\neq0,\ d\in\mathbb{F}_q,\ d\neq c^{q+1},\\[2mm]
0, & b^q=c\neq0,\ d\notin\mathbb{F}_q,\\[2mm]
1-|N'|, & b^q\neq c,\ b^{q+1}=c^{q+1},\
b^qd+cd^q=0,\\[2mm]
2-|N'|, & b^q\neq c,\ b^{q+1}=c^{q+1},\
b^qd+cd^q\neq0,\\[2mm]
2-|N|, & b^{q+1}\neq c^{q+1}.
\end{cases}
\]
Here
\[
N=\left\{i|\,1\leq i\leq q,
h\left(\alpha^{i(q-1)}\right)=0\,\right\},
\]
\[
N'=\left\{i|\,1\leq i\leq q,i\neq k,h\left(\alpha^{i(q-1)}\right)=0\,\right\}.
\]
and
\[
h(x)=(d+bc)^q x^2+
\bigl(\operatorname{Tr}(d)+b^{q+1}+c^{q+1}\bigr)x+d+bc.
\]
In the cases where $b^{q+1}=c^{q+1}, b^q\neq c$, $k$ denotes an integer
with $1\leq k\leq q$ satisfying
\[
b^q\alpha^{kq}+c\alpha^k=0.
\]
\end{lemma}
\begin{proof}
Proceeding as in Lemma~\ref{l6}, we obtain
\begin{align}
Mq^2
&= q^2+\sum_{u\in\mathbb{F}_q^{*}}\chi(du)
\sum_{w\in\mathbb{F}_{q^2}}
\chi\big(uw^{q+1}+(b^q+c)uw\big) \notag\\
&\quad+\sum_{u\in\mathbb{F}_{q^2}\setminus\mathbb{F}_q}
\chi(du)\sum_{w\in\mathbb{F}_{q^2}}
\chi\big(uw^{q+1}+(b^qu^q+cu)w\big).
\label{eq:l7-main}
\end{align}

First suppose that
$b^q\neq c$, $b^{q+1}=c^{q+1}$.
By Lemma~\ref{l2},
\[
\sum_{u\in\mathbb{F}_q^{*}}\chi(du)
\sum_{w\in\mathbb{F}_{q^2}}
\chi\big(uw^{q+1}+(b^q+c)uw\big)=0.
\]
Hence, it remains to evaluate the last sum in
\eqref{eq:l7-main}.

Let
$H=\left\{
u\in\mathbb{F}_{q^2}\setminus\mathbb{F}_q:
b^qu^q+cu=0
\right\}.$
Since every element of $\mathbb{F}_{q^2}\setminus\mathbb{F}_q$ can be uniquely written in the form $\alpha^i u$, where $1\leq i\leq q$ and $u\in\mathbb{F}_q^*$, there is a unique $k$, with $1\leq k\leq q$, such that $
b^q\alpha^{kq}+c\alpha^k=0.
$ Consequently, $H=\left\{\alpha^k u:u\in\mathbb{F}_q^*\right\}.
$

By Lemma~\ref{l2}, we have
\begin{align}
&\sum_{u\in\mathbb{F}_{q^2}\setminus\mathbb{F}_q}
\chi(du)\sum_{w\in\mathbb{F}_{q^2}}
\chi\big(uw^{q+1}+(b^qu^q+cu)w\big) \notag\\
&\qquad =
-q\sum_{u\in H}\chi(du)
-q\sum_{\substack{
u\in\mathbb{F}_{q^2}\setminus\mathbb{F}_q\\
u\notin H}}
\chi\left(
\frac{
\operatorname{Tr}(u)^2du+
u(b^qu^q+cu)^{q+1}
}{
\operatorname{Tr}(u)^2
}
\right)\notag\\
&\qquad=-q\sum_{u\in\mathbb{F}_q^*}
\psi\big(\operatorname{Tr}(d\alpha^k)u\big) \notag
-q\sum_{\substack{1\leq i\leq q\\i\neq k}}
\sum_{u\in\mathbb{F}_q^*}
\psi\big(u\operatorname{Tr}(f(\alpha^i))\big),
\end{align}
where $f(\alpha^i)= \dfrac{d\alpha^i\operatorname{Tr}(\alpha^i)^2+\alpha^i(b^q\alpha^{iq}+c\alpha^i)^{q+1}}{\operatorname{Tr}(\alpha^i)^2}.$\\

Let $N' \subseteq \{1,\dots,q\}\setminus\{k\}$ be such that $ \operatorname{Tr}(f(\alpha^i))=0 \text{ if and only if } i \in N'. $ Therefore, $$ \sum_{u \in \mathbb{F}_q^{*}} \psi\big(u\,\operatorname{Tr}(f(\alpha^i))\big) = \begin{cases} q-1, & \text{if } i \in N',\\[1mm] -1, & \text{if } i \notin N'. \end{cases} $$  
Thus,

$\begin{aligned}
&\sum_{u\in\mathbb{F}_{q^2}\setminus\mathbb{F}_q}
\chi(du)
\sum_{w\in\mathbb{F}_{q^2}}
\chi\bigl(uw^{q+1}+(b^qu^q+cu)w\bigr)
=
\begin{cases}
-q^2|N'|, & b^qd+cd^q=0,\\[2mm]
q^2(1-|N'|), & b^qd+cd^q\neq0.
\end{cases}
\end{aligned}$

Furthermore, $ \operatorname{Tr}(f(\alpha^i))=0 $ if and only if $ (d^q+b^qc^q)x^2+(d+d^q+b^{q+1}+c^{q+1})x+(d+bc)=0, $ where $ x=\alpha^{i(q-1)}$. Thus, the result follow.

The remaining cases can be dealt with in a similar manner.

\end{proof}
In the following lemma, we determine the number of roots of the same polynomial considered in Lemma~\ref{l7} over a field of odd characteristic. We obtain a result that is somewhat similar to that of Lemma~\ref{l7}.

\begin{lemma}\label{l333}
Let $q$ be odd and $\alpha$ be a primitive element of
$\mathbb{F}_{q^2}$. Then $M$, the number of zeroes of $f(x)=x^{q+1}+bx^q+cx+d\in\mathbb{F}_{q^2}[x]$ in $\mathbb{F}_{q^2}$ is given by
\[
M=
\begin{cases}
1, & b=c=d=0,\\[2mm]
q+1, & b=c=0,\ d\in\mathbb{F}_q^*,\\[2mm]
0, & b=c=0,\ d\notin\mathbb{F}_q,\\[2mm]
1, & b^q=c\neq0,\ d\in\mathbb{F}_q,\ d=b^{q+1},\\[2mm]
q+1, & b^q=c\neq0,\ d\in\mathbb{F}_q,\ d\neq b^{q+1},\\[2mm]
0, & b^q=c\neq0,\ d\notin\mathbb{F}_q,\\[2mm]
1-|N'|, &
b^q\neq c,\quad
b^{q+1}=c^{q+1},\quad
b^q d=cd^q,\\[2mm]
2-|N'|, &
b^q\neq c,\quad
b^{q+1}\neq c^{q+1},\quad
b^q d\neq cd^q,\\[2mm]
2-|N|, &
b^{q+1}=c^{q+1}.
\end{cases}
\]
Here
\[
h(x)
=(d-bc)^q x^2
+\bigl(\operatorname{Tr}(d)-b^{q+1}-c^{q+1}\bigr)x
+d-bc,
\]
$k$ and $k'$ are fixed integers satisfying $0<k\leq q, 0\leq k'\leq q,k\neq k',
$
satisfying
\[\operatorname{Tr}(\alpha^k)=0
\text{ and }
b^q\alpha^{k'q}+c\alpha^{k'}=0,
\]
\[
N'
=
\left\{i|0\leq i\leq q\setminus\{k,k'\}:
h\left(\alpha^{i(q-1)}\right)=0
\right\},
\]
and
\[
N
=
\left\{i|0\leq i \leq q\setminus\{k\}:
h\left(\alpha^{i(q-1)}\right)=0
\right\}.
\]
\end{lemma}

\begin{proof}
The proof is routine and follows similarly to the proof of Theorem~\ref{l7}. Applying Lemma~\ref{l3} then yields the stated conclusion
\end{proof}
With all the results concerning the number of zeros in hand, we now proceed to characterize the polynomials as permutation or non-permutation polynomials.

\begin{theorem}\label{t1}
Let $q$ be odd. Then the polynomial
$f(x)=x^q+bx^2+cx+d\in\mathbb{F}_{q^2}[x]$ is not a permutation polynomial of $\mathbb{F}_{q^2}$ if either $b=0$ and $c\in\mu_{q+1}$ or, $b\neq0$.
\end{theorem}

\begin{proof}
First consider the case when $b=0$. Then, by Lemma~\ref{l5}, the result follows directly.

Next assume that $b\neq 0$. Observe that
$
x^q+bx^2+cx+d = g\!\left(x+\dfrac{c}{2b}\right),
$
where
$
g(x)=x^q+bx^2+D,
$
with
$
D=d-\dfrac{c^2}{4b}-\dfrac{c^q}{2b^q}.
$

Since translation preserves the permutation property (and hence the non-permutation property), it suffices to restrict our attention to polynomials of the form $g(x)=x^q+bx^2+D$.
Thus, the equation
$
g(x)+t=0
$
has more than one solution in $\mathbb{F}_{q^2}$ for some $t\in \mathbb{F}_{q^2}$ if and only if
$$
\sum_{i\in K_t} (-1)^i \neq 0,
$$
 where
$
K_t=\left\{\, i \mid 0\leq i\leq q,\; h_t\!\left(\alpha^{\,i(q-1)}\right)=0 \right\},
$
and
$
h_t(x)=b^q x^3-4b^{q+1}(D+t)^q x^2-4b^{q+1}(D+t)x+b.
$
Here, $\alpha$ denotes a primitive element of $\mathbb{F}_{q^2}$. This follows directly from Lemma~\ref{l6}.\\
Notice first that
$
J=\left\{\alpha^{i(q-1)} \mid 0\leq i\leq q \right\}
$
is a subgroup of $\mathbb{F}_{q^2}^{*}$ of order $q+1$, generated by $\alpha^{q-1}$.

We first consider the case $b\in \mathbb{F}_q^{*}$. For
$
t=-D-\dfrac{3}{4b},
$
we have
$
h_t(-1)=0
\quad\text{and}\quad
h_t'(-1)=0.
$
Hence, $-1$ is a root of $h_t(x)=0$ of multiplicity at least $2$. Since $-1\in \mu_{q+1}$ and the product of all roots of $h_t(x)$ is
$
-\dfrac{b}{b^q}=-1,
$
the third root must also be $-1$. Therefore, $-1$ is in fact a root of multiplicity $3$.

Now
$
-1=(\alpha^{q-1})^{\dfrac{q+1}{2}},
$
and therefore
$
K_t=\left\{\dfrac{q+1}{2}\right\}.
$
If $q\equiv 1 \pmod 4$, then $\dfrac{q+1}{2}$ is odd, whereas if $q\equiv 3 \pmod 4$, then $\dfrac{q+1}{2}$ is even. In either case,
$$
\sum_{i\in K_t} (-1)^i \neq 0.
$$
Similarly, assume that $b\in \mathbb{F}_{q^2}\setminus \mathbb{F}_q$ with $\operatorname{Tr}(b)=0$. Take
$
t=\dfrac{2b^{1-q}-1}{4b}-D.
$
Then one readily verifies that $1$ is a root of $h_t(x)=0$ of multiplicity $3$. Consequently,
$
K_t=\{0\}.
$
Therefore,
$$
\sum_{i\in K_t} (-1)^i \neq 0.
$$
Finally, assume that $b\in \mathbb{F}_{q^2}\setminus \mathbb{F}_q$ with $\operatorname{Tr}(b)\neq 0$. For
$
t=\dfrac{1}{4b}-D,
$
the roots of $h_t(x)=0$ are $1$, $-1$, and $\gamma$, where $\gamma\neq ,1-1$. Hence,
$
K_t=\{r,j,k\}.
$

For all possible parity distributions of $r$, $j$ and $k$, we have
$$
\sum_{i\in K_t} (-1)^i \neq 0.
$$
Thus the result follows.

\end{proof}

\begin{remark}\label{r1}
It follows from Theorem~\ref{t1} and Lemma~\ref{l5} that, for odd $q$, a polynomial
$g(x)=x^q+bx^2+cx+d\in\mathbb{F}_{q^2}[x]$ permutes $\mathbb{F}_{q^2}$ if and only if $b=0$ and $c\notin\mu_{q+1}$. In this case, the resulting permutation polynomial is not new, it is of Wan-Lidl type and has been studied extensively in the literature.
\end{remark}
\begin{theorem}\label{tn2}
Let $q=2^k$ and $b,c \in \mathbb{F}_q$.
Then $x^q+bx^2+cx+d\in\mathbb{F}_{q^2}[x]$ is not a permutation polynomial of $\mathbb{F}_{q^2}$ if either $b=0$ and $c=1$, or $
b\neq0$ and $c\neq1$.
\end{theorem}
\begin{proof}
By Lemmas~\ref{l5} and Lemma~\ref{l3.3}, if $b=0$ and $c=1$, or if $b\neq 0$ and $c\neq 1$, then, on taking $t=d$, the equation
$x^q+bx^2+cx+d+t=0$
has more than one solution in $\mathbb{F}_{q^2}$, thus failing to display permutation behavior.
\end{proof}
\begin{remark}\label{r11}
A clear consequence of the above discussion and Lemma~\ref{l5} and Lemma~\ref{l3.3} is that, for even $q$ and
$b,c\in\mathbb{F}_q$, the polynomial
$x^q+bx^2+cx+d\in\mathbb{F}_{q^2}[x]$
permutes $\mathbb{F}_{q^2}$ if and only if either
$b=0$ and $c\neq 1$,
or
$b\neq0$ and $c=1$.
\end{remark}

We note that in Theorem~14 \cite{pp2}, it is shown that when $\operatorname{Tr}(A)=1$, the polynomial $\operatorname{Tr}(A x^{q+1}) + L(x)$ does not permute $\mathbb{F}_{q^2}$. This implies that our polynomial $x^{q+1} + b x^q + c x + d$ never permute $\mathbb{F}_{q^2}$ when $q$ is even. For completeness and to keep the paper self-contained, we also provide an independent proof of this fact i.e., Theorem,~\ref{p3.8}. However, our approach is entirely different from that in \cite{pp2}, thereby preserving the novelty of our method.

\begin{theorem}\label{p3.8}
Let $q$ be even. Then the polynomial
$
x^{q+1}+bx^q+cx+d \in \mathbb{F}_{q^2}[x]
$
is never a permutation polynomial of $\mathbb{F}_{q^2}$.
\end{theorem}
\begin{proof}
For the cases $b=c=0$ and $b^q=c$, choose $t=d+1$ and $t=d$ respectively. By Lemma~\ref{l7}, in each case the resulting polynomial fails to be a permutation polynomial.
Before proceeding further, note that, for any $k$ with $1\leq k\leq q$,

$$
\left\{
\alpha^{i(q-1)}:1\leq i\leq q,\ i\neq k
\right\}
\subsetneq
\left\{
\alpha^{i(q-1)}:1\leq i\leq q
\right\}
=\mu_{q+1}\setminus\{1\}.
$$
\textbf{Case 1} $
b^q\neq c,b^{q+1}=c^{q+1}$.\\
First suppose that $bc\in\mathbb{F}_q$. Taking $t=d+1$, we obtain $h_t(x)=(1+bc)(x^2+1).
$ Thus, $1$ is its only root, and hence $|N'|=0$. By Lemma~\ref{l7}, we have $M=2$. 
Next consider $bc\notin\mathbb{F}_q$. Taking $t=bc+d$, we obtain $
h_t(x)=\operatorname{Tr}(bc)x.
$ Again, $|N'|=0$, and Lemma~\ref{l7} gives $
M=2.
$

Thus, in both sub cases, the polynomial fails to exhibit permutation behavior.\\
\textbf{Case 2:} $
b^{q+1}\neq c^{q+1}.
$\\
If $b=0$, then, taking $t=d$, we obtain $
h_t(x)=c^{q+1}x^2.
$ Thus, $N=0$, and by Lemma~\ref{l7}, $M=2.
$ A similar argument applies when $c=0$. 
Now suppose that $bc\neq0$. Taking $t=d$, we obtain $h_t(x)
=(bc)^qx^2+
\bigl(b^{q+1}+c^{q+1}\bigr)x+bc,
$ which is a self conjugate reciprocal (SCR) polynomial. Notice that $
h_t(1)\neq0.
$ By Lemma~\ref{l4}, the polynomial $h_t(x)$ either has two distinct roots
or has no roots in $\mu_{q+1}\setminus \{1\}$. Hence, $
|N|=0$ or $|N|=2.$ Therefore, $M=2\quad\text{or}\quad M=0,
$ respectively. Thus, in either case, the polynomial under consideration
fails to be a permutation polynomial.

This completes the proof.
\end{proof}
\begin{theorem}
Let $q$ be odd. Then
$
x^{q+1}+bx^q+cx+d \in \mathbb{F}_{q^2}[x]
$
is never a permutation polynomial of $\mathbb{F}_{q^2}$.
\end{theorem}
\begin{proof}
For the cases $b=c=0$ and $b^q\neq0$, take $t=1-d$ and $t=-d$, respectively, and apply Lemma~\ref{l333}. For the remaining cases, choose $
t=bc-d.
$ Then $
h_t(x)=-(b^q-c)^{q+1}x,
$ whose only root is $x=0$. Applying Lemma~\ref{l333} once again gives the desired conclusion, thereby completing the proof.
\end{proof}
In the following proposition, we determine the compositional inverse of the permutation polynomial described in Remark~\ref{r11}. It suffices to consider the case $d=0$, since the general case follows by a translation.

\begin{proposition}
Let $q=2^k$. The compositional inverse of the permutation polynomial in Remark~\ref{r11} is given as follows:
\begin{enumerate}

\item If $b=0$, then
$
f^{-1}(x)=\dfrac{1}{c^2+1}\left(cx + x^q\right).
$
\item If $b\neq 0$, then
$
f^{-1}(x)=\phi^{-1}\!\left(x + \bar{f}^{-1}(\operatorname{Tr}(x))\right),
$
where
$$
\phi^{-1}(x) = \left(b^{-1}x\right)^{2^{2k-1}}
\quad \text{and} \quad
\bar{f}^{-1}(x) = \left(b^{-1}x\right)^{2^{k-1}}.
$$
\end{enumerate}
\end{proposition}
\begin{proof}

(1) The result follows directly from Lemma~\ref{l2.6}.

(2) In Lemma~\ref{l2.7}, take $g=x$, $\phi = b x^2$, which permute $\mathbb{F}_{q^2}$ since $q$ is even, $\psi = \operatorname{Tr}(x)$, $\overline{\psi} = \operatorname{Tr}(x)$, and $h(x)=1$. Then
\[
\psi(\mathbb{F}_{q^2}) = \overline{\psi}(\mathbb{F}_{q^2}) = \mathbb{F}_q.
\]
Hence,
\[
f(x)=b x^2 + x^q + x,
\quad \text{and} \quad
\bar{f}(x)\big|_{\mathbb{F}_q}= b x^2.
\]
Therefore,
\[
\bar{f}^{-1}(x) = (b^{-1}x)^{2^{k-1}}
\quad \text{and} \quad
\phi^{-1}(x)= (b^{-1}x)^{2^{2k-1}}.
\]
Thus, the result follows.

\end{proof}

\section{Conclusion}
We investigate the permutation behavior of polynomials of the forms $x^q+bx^2+cx+d$ and $x^{q+1}+bx^q+cx+d
$ over $\mathbb{F}_{q^2}$ using character sums, particularly Weil sums. Furthermore, by employing self-conjugate-reciprocal polynomials, we establish several general classes of polynomials that do not permute $\mathbb{F}_{q^2}$. To be precise, we conclude that the polynomial $
x^{q+1}+bx^q+cx+d
$ never permutes $\mathbb{F}_{q^2}$, regardless of whether the characteristic is odd or even. The conclusions for the other family of polynomials are somewhat different.

For future research, one may further investigate the general strategy of determining the number of roots to characterize permutation behavior and apply it to other classes of polynomials.

	\end{document}